\documentclass[12pt]{article}
\usepackage{amsmath}
\usepackage{amssymb}
\usepackage{amsxtra}
\usepackage{amscd}
\usepackage{amsthm}
\usepackage{tikz}

\DeclareMathOperator{\area}{area}
\DeclareMathOperator{\bounce}{bounce}
\DeclareMathOperator{\dinv}{dinv}
\DeclareMathOperator{\Hilb}{Hilb}
\DeclareMathOperator{\Sing}{Sing}
\DeclareMathOperator{\mult}{mult}
\DeclareMathOperator{\Mod}{Mod}
\DeclareMathOperator{\Id}{Id}

\newcommand{\jc}{\overline{JC}}
\newcommand{\ophi}{\overline{\varphi_1}}
\newcommand{\gpq}{\Gamma^{p,q}}
\newcommand{\gnn}{\Gamma^{n,n+1}}
\newcommand{\hpq}{h^{+}_{\frac{p}{q}}}
\newcommand{\hnn}{h^{+}_{\frac{n}{n+1}}}
\newcommand{\rpq}{R_{+}^{p,q}}
\newcommand{\rnn}{R_{+}^{n,n+1}}
\newcommand{\ZZ}{\mathbb{Z}_{\ge 0}}

\newtheorem{lemma}{Lemma}[section]
\newtheorem{example}[lemma]{Example}
\newtheorem{theorem}[lemma]{Theorem}
\newtheorem{corollary}[lemma]{Corollary} 
\newtheorem{conjecture}[lemma]{Conjecture} 
\theoremstyle{definition}
\newtheorem{definition}[lemma]{Definition}
\newtheorem{remark}[lemma]{Remark}

\title{Compactified Jacobians and $q,t$-Catalan Numbers, I}

\author{Evgeny Gorsky\footnote{Department of Mathematics, SUNY at Stony Brook, Stony Brook, NY 11794-3651} \footnote{Laboratoire J.-V. Poncelet (UMI 2615). 119002, Bolshoy Vlasyevskiy Pereulok 11, Moscow.}, 
Mikhail Mazin\footnote{IMS, SUNY at Stony Brook, Stony Brook, NY 11794-3660}}

\date{}

\begin{document}
\maketitle 
 
\begin{abstract}
J. Piontkowski described the homology of the Jacobi factor of a plane curve singularity with one
Puiseux pair. We discuss the combinatorial structure of his answer, in particular, relate it to
the bigraded deformation of Catalan numbers introduced by A. Garsia and M. Haiman. 

\end{abstract}


\section{Introduction}

The compactified Jacobians of singular curves were introduced by C. Rego in \cite{rego}. Some of their general properties were established by A. Altman, S. Kleiman and A. Iarrobino in \cite{aik} and \cite{kleiman}. In particular, the compactified Jacobian is irreducible if and only if all curve singularities have embedding dimension 2.
For a rational unibranched curve $C$ its compactified Jacobian is homeomorphic to the direct product of compact spaces, the Jacobi factors
 $\jc_p$, $p\in \Sing(C)$, which depend only on the analytic type of the singularities of $C$ \cite{beauville}. Following the ideas of A. Beauville, S. T. Yau and E. Zaslow (\cite{beauville}, \cite{yz}), L. Goettsche, B. Fantechi and D. van Straten showed in \cite{fgs} the significance of Jacobi factors in enumerative geometry: they related the Euler characteristic of the Jacobi factor $\jc_p$ to the multiplicity of the $\delta$-constant strata in the base of the miniversal deformation of $(C,p)$. This relation was generalized by V. Shende in \cite{shende} where he proved analogous relations between Hilbert schemes of points on a singular curve and multiplicities of some strata in the base of the miniversal deformation.

J. Piontkowski in \cite{piont2} (see also \cite{piont}) proved that in some cases the Jacobi factors can be decomposed into explicit affine cells. In particular, for a plane curve singularity with one Puiseux pair $(p,q)$ he proved that these cells can be described as follows.

\begin{definition}
Let $\gpq:=\{ap+bq :  a,b \in\mathbb Z_{\ge 0}\}$ be the semigroup generated by $p$ and $q.$ A subset $\Delta\subset\ZZ$ is called a $0$-normalized $\gpq$--semi-module if $0\in\Delta$ and $\Delta+\gpq\subset\Delta.$
\end{definition}


The cells of the Jacobi factor are parametrised by all possible $0$--normalized $\gpq$--semi-modules $\Delta,$ and the dimension of the cell $C_{\Delta},$ corresponding to a semi-module $\Delta,$  is a combinatorial invariant of $\Delta.$  This gives a complete description of the homology of the Jacobi factor in this case.

We give an alternative combinatorial description of Piontkowski's answer.

\begin{definition}(\cite{lowa})\label{defhplus}
Let $D$ be a Young diagram, $c\in D$. Let $a(c)$ and $l(c)$ denote the lengths of arm and leg for $c$.
For each real nonnegative $x$  define
$$h^{+}_{x}(D)=~\vline \left\lbrace c\in D:\frac{a(c)}{l(c)+1}\le x< \frac{a(c)+1}{l(c)}\right\rbrace \vline .$$
\end{definition}

\begin{remark}
If $l(c)=0,$ we say that $\frac{a(c)+1}{l(c)}=\infty>x.$
\end{remark}

Let $R^{p,q}$ be a $(p,q)$-rectangle, where $(p,q)$ are coprime. Let $\rpq\subset R^{p,q}$ be the subset
consisting of boxes which lie below the left-top to right-bottom diagonal.

We construct a natural bijection $D$ between the set of  $\gpq$--semi-modules and the set of Young diagrams contained in $\rpq$, such that
\begin{equation}
\label{T1}
|D(\Delta)|=|\Delta\setminus\gpq|,\quad \dim C_{\Delta}=\frac{(p-1)(q-1)}{2}-\hpq(D(\Delta)).
\end{equation}

In \cite{gaha} A. Garsia and M. Haiman constructed a sequence of bivariate generalizations of the Catalan numbers $C_{n}(q,t)$.
The combinatorial description of these polynomials was obtained by A. Garsia and J. Haglund in \cite{gahaim}.
Equation (\ref{T1}) together with the results of \cite{gahaim} allows us to prove the following equation:

\begin{equation}
\sum_{\Delta\in \Mod_{\gnn}}q^{\dim(C_{\Delta})}t^{|\ZZ\setminus \Delta|}=q^{\binom{n}{2}}C_{n}(q^{-1},t).
\end{equation}
\noindent where $\Mod_{\gnn}$ is the set of $0$-normalized $\gnn$--semi-modules.

Therefore, one obtains the following formula for the Poincar\'e polynomial of the Jacobi factor of the singularity with one Puiseux pair $(n,n+1):$
\begin{equation}
\label{Pnnp1}
P_{n,n+1}(t)=t^{2\binom{n}{2}}C_{n}(t^{-2},1)=\sum_{D\subset \rnn}t^{2|D|}.
\end{equation}

 We also construct a map $G$ from the set of diagrams inscribed in $\rpq$ to itself, such that for all $\Delta$
$$|G(D(\Delta))|=\dim C_{\Delta}.$$

We conjecture that the map $G$ is a bijection from the set of subdiagrams of $\rpq$ to itself. In particular, it would imply
a direct generalization of the formula (\ref{Pnnp1}) to the case of the Jacobi factor of a singularity with one Puiseux pair $(p,q)$ for arbitrary coprime $(p,q):$
\begin{equation}
P_{p,q}(t)=\sum_{D\subset \rpq}t^{2\dim\Delta(D)}=\sum_{D\subset \rpq}t^{2|G(D)|}=\sum_{D\subset \rpq}t^{2|D|}
\end{equation}

We prove this conjecture for the $(n,n+1)$ case in Theorem \ref{nnp1}. It turns out that in this case the map $G$ gives an alternative description
of the combinatorial bijection considered by J. Haglund \cite{haglund}. 

Recently A. Buryak \cite{buryak} provided an example of a smooth subvariety of the Hilbert scheme of points in $\mathbb{C}^2$ which admits a decomposition into affine cells enumerated by Dyck paths, such that the dimensions of these cells are related to the $h^{+}$ statistic. In  Section \ref{Mpqh}, we construct an explicit bijection between the cell decompositions from \cite{piont2}
and \cite{buryak}.

In Section \ref{sec:dimension} we construct a bijection
between Piontkowski's cells and Young diagrams inscribed in a right
$(p,q)$-triangle, and prove a formula for the dimensions of cells in
terms of the diagrams. In Section \ref{sec:poly} we explore a relation between
our combinatorial description of  the Piontkowski's cell decomposition
and the $q,t$-Catalan numbers and their generalizations. In Section
\ref{Mpqh} we relate  Piontkowski's cell decomposition to a cell
decomposition of an open subvariety in the Hilbert scheme of points in
$\mathbb{C}^2,$ introduced in \cite{buryak}. Finally, in Section
\ref{sec:knots} we mention a possible relation of our results to HOMFLY
homology of torus knots.

We continue our study of combinatorics of Jacobi factors in \cite{GM}.

\section{Compactified Jacobians and Jacobi factors}
\label{sec:dimension}

\subsection{Preliminary Information.}

 We define the compactified Jacobian following \cite{piont2} and \cite{rego} (see also \cite{gk}, \cite{gp}).
Let $C$ be a possibly singular complete algebraic curve. 

\begin{definition}
The Jacobian $JC$ of $C$ consists of the locally free sheaves of rank 1
and degree 0 on $C$.
\end{definition}

\begin{definition}
The compactified Jacobian $\jc$ of $C$ consists of the torsion free sheaves of rank 1
and degree 0 on $C$, i.e. $\chi(F)=1-g_{a}(C)$.
\end{definition}

Here is list of useful facts about the compactified Jacobians.

\begin{theorem}(\cite{rego},\cite{aik},\cite{kleiman})
If all singularities of $C$ have embedding dimension $2$, then $\jc$ is irreducible.
If $C$ has a singularity of embedding dimension $\ge 3$, then $\jc$ is reducible.
\end{theorem}

\begin{theorem}(\cite{rego})
Suppose that $C$ has only planar singularities. Then the boundary $\jc\setminus JC$
is a union of $m$ irreducible codimension one subsets of $\jc$ where
$$m=\sum_{p\in C}(\mult\mathcal{O}_{C,p}-1).$$
\end{theorem}

\begin{theorem}(\cite{beauville})
For a rational unibranched curve $C,$ its compactified Jacobian is homeomorphic to the direct product of compact spaces, the Jacobi factors
 $\jc_p$, $p\in \Sing(C)$, which depend only on the analytic type of the singularities $(C,p)$.
\end{theorem}

If an irreducible plane curve singularity has one Puiseux pair $(p,q)$, its Jacobi factor admits a natural cell decomposition (\cite{piont}, \cite{piont2}).

Consider the semi-group $\gpq\subset\ZZ$ and a $0$--normalized $\gpq$--semi-module $\Delta.$

\begin{definition}\label{p-basis}
The {\it $p$--basis} $\{0=a_0<a_1<\dots<a_{p-1}\}$ of $\Delta$ is the set consisting smallest integers in each congruence class modulo $p.$ In other words $\{a_0,\dots,a_{p-1}\}=\Delta\setminus (\Delta+p).$ The elements of the $p$--basis are called {\it $p$--generators}. We always put $p$--generators in the increasing order $0=a_0<a_1<\dots<a_{p-1}.$
\end{definition}

\begin{definition}(\cite{piont}) 
\label{def of dim}
The dimension of a semi-module $\Delta$ is defined as
$$\dim \Delta=\sum_{j=0}^{p-1}\left|[a_j, a_j + q)\setminus \Delta\right| .$$
\end{definition}

\begin{theorem}(\cite{piont})
\label{pi}
The Jacobi factor of a singularity with one Puiseux pair $(p,q)$ admits a natural cell decomposition with affine cells $C_{\Delta},$
parametrised by the 0-normalized $\gpq$--semi-modules $\Delta.$ The dimension of the cell $C_{\Delta}$ equals 
$\dim \Delta$.
\end{theorem}

\subsection{Combinatorics of Piontkowski's Decomposition.}

In this section the coprime integers $(p,q)$ will be fixed, so we drop the superscript $(p,q)$ in the notations for the semigroup $\Gamma=\gpq,$ the $p\times q$ rectangle $R=R^{p,q},$ and the subset $R_{+}=\rpq$ of boxes of $R,$ which lie under the left-top to right-bottom diagonal.

We will parametrise these cells of Piontkowski's decomposition by certain Young diagrams. Let us label the boxes of $R$ 
with integers, so that the shift by $1$ up subtracts
$p,$ and the shift by $1$  to the right subtracts $q.$ We normalize these numbers so that
$pq$ is in the box $(0,0)$ (note that this box is not in the rectangle $R,$ as we start enumerating
boxes from $1$). In other words, the numbers are given by the linear function $f(x,y)=pq-qx-py.$

One can see that the labels of the boxes of $R_+$ are positive, while all other labels in
$R$ are negative. Moreover, numbers in the boxes of $R_+$ are exactly the numbers from the
complement $\ZZ\backslash\Gamma,$ and each such number appears only once in $R_+.$ In
particular, the area of $R_+$ is equal to $\delta=\frac{(p-1)(q-1)}{2}.$

\begin{definition}
For a  $0$-normalized $\Gamma$--semi-module $\Delta$, let $D(\Delta)$ denote the set of boxes with labels belonging to
$\Delta\setminus\Gamma$.
\end{definition}

\begin{remark}
$D(\Delta)$ is a Young diagram of area $|\Delta\setminus\gpq|$.
The $p$--basis of $\Delta$ consists of numbers on the top of columns of $D(\Delta)$. If the corresponding column is empty, one should take the number below this column. The correspondence between $\Delta$ and $D(\Delta)$ is bijective.
\end{remark}

We illustrate the definition of the diagram $D(\Delta)$ in Figure \ref{5times7}.

\begin{figure}[ht]
\begin{center}
\begin{tikzpicture}
\fill [color=lightgray] (0,0)--(0,5)--(1,5)--(1,4)--(2,4)--(2,2)--(3,2)--(3,1)--(4,1)--(4,0)--(0,0);
\fill [color=gray] (0,0)--(0,4)--(1,4)--(1,2)--(2,2)--(2,0)--(0,0);
\draw (0,0) grid (5,7);
\draw [dashed] (2,-1) grid (5,0);
\draw (2.5,-0.5) node {$14$};
\draw (3.5,-0.5) node {$7$};
\draw (4.5,-0.5) node {$0$};
\draw (0.5,0.5) node {$23$};
\draw (1.5,0.5) node {$16$};
\draw (2.5,0.5) node {$9$};
\draw (3.5,0.5) node {$2$};
\draw (4.5,0.5) node {$-5$};
\draw (0.5,1.5) node {$18$};
\draw (1.5,1.5) node {$11$};
\draw (2.5,1.5) node {$4$};
\draw (3.5,1.5) node {$-3$};
\draw (4.5,1.5) node {$-10$};
\draw (0.5,2.5) node {$13$};
\draw (1.5,2.5) node {$6$};
\draw (2.5,2.5) node {$-1$};
\draw (3.5,2.5) node {$-8$};
\draw (4.5,2.5) node {$-15$};
\draw (0.5,3.5) node {$8$};
\draw (1.5,3.5) node {$1$};
\draw (2.5,3.5) node {$-6$};
\draw (3.5,3.5) node {$-13$};
\draw (4.5,3.5) node {$-20$};
\draw (0.5,4.5) node {$3$};
\draw (1.5,4.5) node {$-4$};
\draw (2.5,4.5) node {$-11$};
\draw (3.5,4.5) node {$-18$};
\draw (4.5,4.5) node {$-25$};
\draw (0.5,5.5) node {$-2$};
\draw (1.5,5.5) node {$-9$};
\draw (2.5,5.5) node {$-16$};
\draw (3.5,5.5) node {$-23$};
\draw (4.5,5.5) node {$-30$};
\draw (0.5,6.5) node {$-7$};
\draw (1.5,6.5) node {$-14$};
\draw (2.5,6.5) node {$-21$};
\draw (3.5,6.5) node {$-28$};
\draw (4.5,6.5) node {$-35$};
\draw [dashed] (0,7)--(5,0);
\draw [line width=1] (0,0)--(0,4)--(1,4)--(1,2)--(2,2)--(2,0)--(0,0);
\end{tikzpicture}
\caption{Here $p=5,\ q=7,$ and $\Delta=\ZZ\backslash \{1,2,3,4,6,9\}=\{0,5,7,8,10,11,\dots\}.$ The $p$--basis is $\{0,7,8,11,14\}.$}\label{5times7}
\end{center}
\end{figure}

We are ready to give the description of dimensions of the cells $C_{\Delta}$ in terms of the corresponding
Young diagrams $D(\Delta)$. To shorten the notations, we use $D=D(\Delta).$

\begin{theorem}
\label{piontcell}
The dimensions of cells can be expressed through the $h^{+}$ statistic:
$$\dim C_{\Delta}=\frac{(p-1)(q-1)}{2}-\hpq(D(\Delta)).$$
\noindent (See Definition \ref{defhplus} for the definition of $\hpq.$)
\end{theorem}

\begin{proof}
The proof is based on Lemma \ref{UV}, proved in Section \ref{Section UV}. Here we define subsets $U(D),V(D)\subset R,$ such that $|U(D)|=\dim C_{\Delta},$ and $|V(D)|=(p-1)(q-1)/2-\hpq(D).$ In Lemma \ref{UV} we construct a bijection between $U(D)$ and $V(D).$

The definition of the set $V$ is straightforward. According to the definition, $\hpq$ counts the boxes $c\in D(\Delta),$ such that $\frac{a(c)}{l(c)+1}\le \frac{p}{q}< \frac{a(c)+1}{l(c)}.$ On the other side, we know that $|R_+|=\frac{(p-1)(q-1)}{2}.$

\begin{definition}
Define the sets
$$H_{\frac{p}{q}}(D)=\left\{c\in D: \frac{a(c)}{l(c)+1}\le\frac{p}{q}<\frac{a(c)+1}{l(c)}\right\},$$
$$V(D)=R_+\backslash H_{\frac{p}{q}}(D).$$
\end{definition}

\begin{remark}
The equality $\frac{a(c)}{l(c)+1}=\frac{p}{q}$ in the definition of $H_{\frac{p}{q}}(D)$ is never attained, because $p$ and $q$ are coprime, and $a(c)<p$ and $l(c)<q.$
\end{remark}

In order to define the set $U(D)$ we will need to provide another way of computing the dimension of the cell $C_{\Delta}.$ We will need the following definition:

\begin{definition}
We call an integer $y$ a $q$-cogenerator of  $\Delta$ if $y\not\in \Delta$ and $y+q\in \Delta$.
\end{definition}
\begin{remark}
$q$-cogenerators are the labels on the leftmost boxes of the rows of the complement $R\backslash D.$ In particular, in the example in Figure \ref{5times7} the $q$--cogenerators are $\{-7,-2,1,3,4,6,9\}.$
\end{remark}

\begin{lemma}
The number of $q$-cogenerators of  $\Delta$ greater or equal to $a$ equals  $g(a):=|[a,a+q)\setminus \Delta|$.
\end{lemma}

\begin{proof}
Consider the arithmetic sequences with difference $q$ starting from the elements of $[a, a + q)\setminus \Delta$.
They are pairwise disjoint and each of them contains exactly one $q$-cogenerator.
\end{proof}

\begin{corollary}
The dimension of a cell $C_{\Delta}$ equals  the number of pairs $(a_i,b_j)$ where $a_i$ is a $p$-generator of $\Delta$,
$b_j$ is a $q$-cogenerator of $\Delta,$ and $a_i<b_j$.
\end{corollary}

Each column of the rectangle $R$ contains exactly one $p$-generator, and each row of $R$ contains exactly one $q$-cogenerator. Therefore, there is a natural bijection between the boxes of $R$ and the couples $(a_i,b_j),$ where $a_i$ is a $p$-generator of $\Delta$ and $b_j$ is a $q$-cogenerator of $\Delta$.

\begin{definition}
We define the set $U(D)\subset R$ to be the set of boxes, such that the corresponding couple $(a_i,b_j)$ satisfies $a_i<b_j.$
\end{definition}

One can reformulate this definition in terms of arms and legs of boxes. Before doing so, we need to extend the definitions of arms and legs to the boxes in $R\backslash D:$

\begin{definition}
Let $c=(x,y)\in R\backslash D.$ We define the leg and arm of $c$ by the equations
$$l(c)=\max\left\{n: (x,y-n)\notin D\right\},\quad a(c)=\max\left\{n: (x-n,y)\notin D\right\}.$$
\end{definition}

One immediately gets the following Lemma:

\begin{lemma}\label{U}
The set $U(D)$ can be described in terms of the $a(c)$ and $l(c)$ as follows:
$$
U(D)=\{c\in R\backslash D: a(c)q>(l(c)+1)p\}\cup \{c\in D: (a(c)+1)q\le l(c)p\},
$$
\end{lemma}

\begin{remark}
As above, the equality $(a(c)+1)q=l(c)p$ is never attained.
\end{remark}

We illustrate Lemma \ref{U} as well as the definitions of arms and legs in Figure \ref{legarm}.

\begin{figure}[ht]
\begin{center}
\begin{tikzpicture}
\draw (0,0)--(0,7)--(5,7)--(5,0)--(0,0);
\draw (0,5)--(1,5)--(1,3)--(2,3)--(2,2)--(3,2)--(3,0.5)--(4,0.5)--(4,0);
\draw (1.3,0.7) rectangle (1.7,1.1);
\draw (3,0.7) rectangle (3.4,1.1);
\draw (1.3,2.6) rectangle (1.7,3);
\draw [<->,>=stealth] (1.7,0.9)--(3,0.9);
\draw [<->,>=stealth] (1.5,1.1)--(1.5,3);
\draw (2.3,0.6) node {\small $a(c)$};
\draw (1.2,2) node {\small $l(c)$};
\draw (0.5,3) node {$D$};
\draw [->,>=stealth] (1.2,0.6)--(1.5,0.9);
\draw [->,>=stealth] (3.5,1.2)--(3.2,0.9);
\draw [->,>=stealth] (2,3.2)--(1.6,2.8);
\draw (0.9,0.4) node {\small $f(c)$};
\draw (3.7,1.2) node {\small $b_j$};
\draw (2.2,3.2) node {\small $a_i$};

\draw (6,0)--(6,7)--(11,7)--(11,0)--(6,0);
\draw (6,5)--(7,5)--(7,3)--(8,3)--(8,2)--(9,2)--(9,0.5)--(10,0.5)--(10,0);
\draw (8.3,3.7) rectangle (8.7,4.1);
\draw (8.3,2) rectangle (8.7,1.6);
\draw (7,3.7) rectangle (7.4,4.1);
\draw [<->,>=stealth] (7,3.9)--(8.3,3.9);
\draw [<->,>=stealth] (8.5,2)--(8.5,3.7);
\draw (7.8,4.1) node {\small $a(c)$};
\draw (8.9,2.9) node {\small $l(c)$};
\draw (6.5,1) node {$D$};
\draw [->,>=stealth] (8.8,4.2)--(8.5,3.9);
\draw [->,>=stealth] (8.1,1.4)--(8.4,1.7);
\draw [->,>=stealth] (6.8,3.5)--(7.1,3.8);
\draw (9.1,4.4) node {\small $f(c)$};
\draw (6.7,3.4) node {\small $b_j$};
\draw (8,1.4) node {\small $a_i$};
\end{tikzpicture}
\caption{\small For a box $c\in D$ one has $a_i=f(c)-l(c)p$ and $b_j=f(c)-(a(c)+1)q;$ for $c\in R\backslash D$ one gets $a_i=f(c)+(l(c)+1)p$ and $b_j=f(c)+a(c)q.$}\label{legarm}
\end{center}
\end{figure}
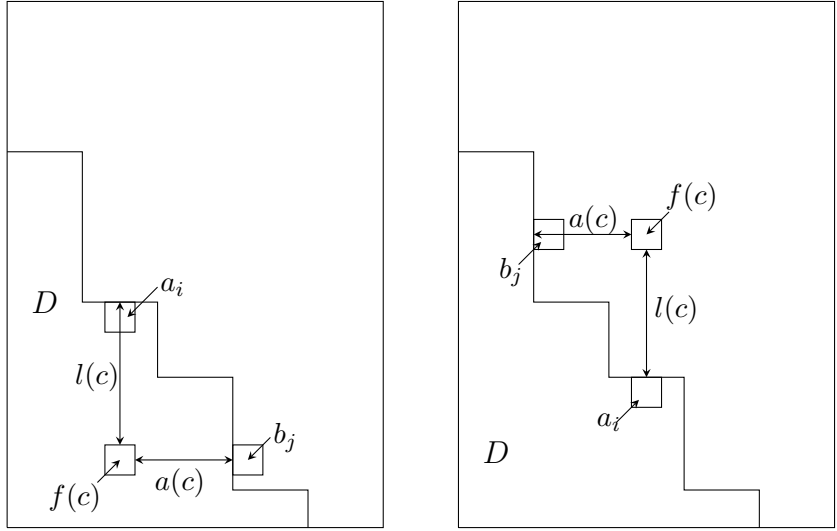

Now  Theorem \ref{piontcell} follows from Lemma \ref{UV}.
\end{proof}

\subsection{Proof of the dimension formula}\label{Section UV}


\begin{lemma}
\label{UV}
There exist a natural bijection between the sets $U(D)$ and $V(D)$.
\end{lemma}

\begin{proof}

Note that $V(D)$ naturally splits into $3$ pieces

$$V=V_1\sqcup V_2\sqcup V_3,$$
where
$V_1=R_+\backslash D,$ $$V_2=\left\{c\in D:\frac{a(c)+1}{l(c)}\le\frac{p}{q}\right\},
\ \mbox{\rm and}\ 
V_3=\left\{ c\in D: \frac{a(c)}{l(c)+1} > \frac{p}{q}\right\} .
$$


Note that $V_2=U\cap D.$ This suggests that the set $U$ should also split into $3$ pieces $U=U_1\sqcup
U_2\sqcup U_3,$ so that $U_2=V_2$ and there are bijections $\varphi_1:U_1\to V_1,$ and $\varphi_3:U_3\to
V_3.$

Let
$$
U_1=\{c=(x,y)\in R\backslash D: a(c)q\ge yp\},\quad
U_2=U\cap D,
$$
and
$$
U_3=\{c=(x,y)\in R\backslash D: (l(c)+1)p<a(c)q<yp\}.
$$

{\bf Bijection $\varphi_1.$} Consider the map $\ophi:R\backslash D\to R\backslash D$ which preserves the rows and inverts the order in each row of $R\backslash D$ (see Figure \ref{phi1}).

\begin{figure}[ht]
\begin{center}
\begin{tikzpicture}
\draw (0,0)--(0,7)--(5,7)--(5,0)--(0,0);
\draw (0,5)--(1,5)--(1,3)--(2,3)--(2,2)--(3,2)--(3,0.5)--(4,0.5)--(4,0);
\draw [dashed] (0,7)--(5,0);
\draw (1,3.5) rectangle (5,3.7);
\draw (1.7,3.5) rectangle (1.9,3.7);
\draw (4.3,3.5) rectangle (4.1,3.7);
\draw [<->,>=stealth] (1,3.6)--(1.7,3.6);
\draw [<->,>=stealth] (4.3,3.6)--(5,3.6);
\draw (3,1.2) rectangle (5,1.4);
\draw (3.3,1.2) rectangle (3.5,1.4);
\draw (4.7,1.2) rectangle (4.5,1.4);
\draw [<->,>=stealth] (3,1.3)--(3.3,1.3);
\draw [<->,>=stealth] (4.7,1.3)--(5,1.3);
\draw [<->,>=stealth] (1.8,3.7) .. controls (3,4) and  (3,4) .. (4.2,3.7);
\draw [<->,>=stealth] (3.4,1.4) .. controls (4,1.7) and  (4,1.7) .. (4.6,1.4);
\draw (3.1,4.2) node {\small $\ophi$};
\draw (4.1,1.8) node {\small $\ophi$};
\draw (0.5,1) node {$D$};
\end{tikzpicture}
\caption{The map $\ophi$ switches the order in each row of $R\backslash D$.}\label{phi1}
\end{center}
\end{figure}

In other words, $\ophi$ is given by
$$
\ophi(c)=(p-a(c),y)
$$
\noindent where $c=(x,y).$

We set $\varphi_1={\ophi}|_{U_1}.$ Note that ${\ophi}^2=\Id_{R\backslash D}.$ Therefore, one only needs to prove that $c\in U_1$ is equivalent to $\ophi(c)\in V_1=R_+\backslash D.$ Indeed,
$$
a(c)q\ge yp\Leftrightarrow pq-q(p-a(c))-py\ge 0\Leftrightarrow (p-a(c),y)\in R_+,
$$
\noindent and $(p-a(c),y)\notin D$ by the definition of $a(c).$

{\bf Bijection $\varphi_3.$} The bijection $\varphi_3:U_3\to V_3$ is slightly more tricky.

For $c=(x,y)\in U_3,$ let
$
m(c)= \left\lfloor \frac{a(c)q}{p}\right\rfloor.
$
Let $y'=y-m(c).$ Since $c\in U_3,$ it follows that $(x,y')\in D.$ Indeed,
$
m(c)=\left\lfloor \frac{a(c)q}{p}\right\rfloor\ge l(c)+1
$
since $(l(c)+1)p\le a(c)q.$

Also $m(c)<y,$ because $a(c)q\le yp$ and $\frac{a(c)q}{p}$ cannot be an integer ($a(c)<p,$ and $(p,q)$ are coprime). Therefore, $y-l(c)-1\ge y-m(c)=y'>0.$

Let $x'=x+a(x,y')-a(x,y).$ 
Set $\varphi_3(c)=c':=(x',y').$ (See Figure \ref{phi3} for the illustration.)

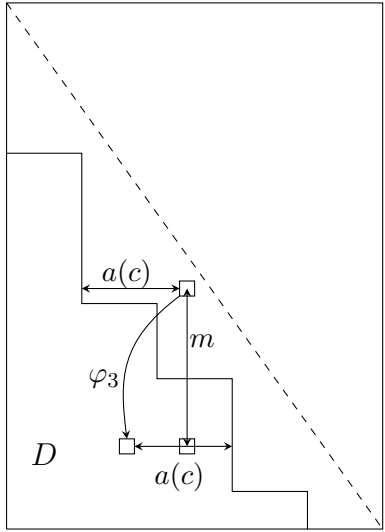
\begin{figure}[ht]
\begin{center}
\begin{tikzpicture}
\draw (0,0)--(0,7)--(5,7)--(5,0)--(0,0);
\draw (0,5)--(1,5)--(1,3)--(2,3)--(2,2)--(3,2)--(3,0.5)--(4,0.5)--(4,0);
\draw [dashed] (0,7)--(5,0);
\draw (2.3,3.1) rectangle (2.5,3.3);
\draw (2.3,1) rectangle (2.5,1.2);
\draw (1.7,1) rectangle (1.5,1.2);
\draw [<->,>=stealth] (1,3.2)--(2.3,3.2);
\draw [<->,>=stealth] (2.4,1.1)--(2.4,3.2);
\draw [<->,>=stealth] (1.7,1.1)--(3,1.1);
\draw [->,>=stealth] (2.3,3.1).. controls (1.5,2.5) and (1.5,2) ..(1.6,1.2);
\draw (1.6,3.4) node {\small $a(c)$};
\draw (2.3,0.7) node {\small $a(c)$};
\draw (2.6,2.5) node {\small $m$};
\draw (1.3,2) node {\small $\varphi_3$};
\draw (0.5,1) node {$D$};
\end{tikzpicture}
\caption{The map $\varphi_3.$}\label{phi3}
\end{center}
\end{figure}

Since $c\in U_3$, we have
$$
a(c')=a(c)\quad \mbox{\rm and}\quad
l(c')+1 \le m(c),
$$
Therefore
$$
\frac{a(c')}{l(c')+1}\ge a(c)/\left\lfloor \frac{a(c)q}{p}\right\rfloor\ge \frac{p}{q}.
$$

We saw before that $\frac{a(c)q}{p}$ is not an integer. Therefore, the second inequality is actually strict. So, $c'\in V_3.$

One can check that the map $\varphi_3$ is invertible. Indeed, given $c'=(x',y')\in V_3,$ one set $m(c')=\left\lfloor
\frac{a(c')q}{p}\right\rfloor.$ Then $c=(x,y)=\varphi_3^{-1}(c')$ is uniquely determined by $y=y'+m(c')$ and
$a(c)=a(c').$

Since
$$l(c)+1\le m(c')=\left\lfloor
\frac{a(c')q}{p}\right\rfloor=\left\lfloor\frac{a(c)q}{p}\right\rfloor<\frac{a(c)q}{p}\quad \mbox{\rm and}\quad $$
$$y=y'+m(c')>m(c')=\left\lfloor\frac{a(c')q}{p}\right\rfloor\Longrightarrow y>\frac{a(c')q}{p}=\frac{a(c)q}{p},$$

The box $c$ belongs to the set $U_3$ (note that we again used that $\frac{a(c)q}{p}\notin\mathbb Z$).

\end{proof}

\begin{example}
Let $(p,q)=(5,6).$ Consider the diagram $D$ consisting of two columns of height $3.$ We illustrate the bijections $\varphi_i$ in  Figure \ref{exampleUV}.

\begin{figure}[ht]
\begin{center}
\begin{tikzpicture}
\draw [step=0.5, color=gray!50!white, very thin] (0,0) grid (2.5,3);
\filldraw [fill=gray!60!white] (0.5,1.5)--(1,1.5)--(1,2)--(0.5,2)--(0.5,1.5);
\draw (0.75,1.75) node {$3$};
\filldraw [fill=gray!15!white] (0,1)--(0.5,1)--(0.5,1.5)--(0,1.5)--(0,1);
\draw [->,>=stealth,very thin] (0.75,1.75).. controls (0.7,1.6) and (0.4,1.3) ..(0.25,1.25);

\filldraw [fill=gray!60!white] (0.5,0)--(1,0)--(1,0.5)--(0.5,0.5)--(0.5,0);
\draw (0.75,0.25) node {$2$};

\filldraw [fill=gray!60!white] (2,1.5)--(2.5,1.5)--(2.5,2)--(2,2)--(2,1.5);
\draw (2.25,1.75) node {$1$};
\filldraw [fill=gray!15!white] (0,1.5)--(0.5,1.5)--(0.5,2)--(0,2)--(0,1.5);
\draw [->,>=stealth,very thin] (2.25,1.85).. controls (1.75,2.25) and (0.75,2.25) ..(0.25,1.85);
\filldraw [fill=gray!60!white] (2,0.5)--(2.5,0.5)--(2.5,1)--(2,1)--(2,0.5);
\draw (2.25,0.75) node {$1$};
\filldraw [fill=gray!15!white] (1,0.5)--(1.5,0.5)--(1.5,1)--(1,1)--(1,0.5);
\draw [->,>=stealth,very thin] (2.25,0.85).. controls (2,1.2) and (1.5,1.2) ..(1.25,0.85);
\filldraw [fill=gray!60!white] (1.5,0)--(2,0)--(2,0.5)--(1.5,0.5)--(1.5,0);
\draw (1.75,0.25) node {$1$};
\filldraw [fill=gray!60!white] (2,0)--(2.5,0)--(2.5,0.5)--(2,0.5)--(2,0);
\draw (2.25,0.25) node {$1$};
\filldraw [fill=gray!15!white] (1,0)--(1.5,0)--(1.5,0.5)--(1,0.5)--(1,0);
\draw [->,>=stealth,very thin] (2.25,0.35).. controls (2,0.7) and (1.5,0.7) ..(1.25,0.35);

\draw [dashed, gray!80!white] (0,3)--(2.5,0);
\draw [very thick] (0,0) rectangle (2.5,3);
\draw [very thick] (0,0) rectangle (1,1.5);

\draw (1.25,-0.4) node {$U(D)$};

\draw [step=0.5, color=gray!50!white, very thin] (5,0) grid (7.5,3);

\filldraw [fill=gray!60!white] (5,1)--(5.5,1)--(5.5,1.5)--(5,1.5)--(5,1);
\draw (5.25,1.25) node {$3$};

\filldraw [fill=gray!60!white] (5.5,0)--(6,0)--(6,0.5)--(5.5,0.5)--(5.5,0);
\draw (5.75,0.25) node {$2$};

\filldraw [fill=gray!60!white] (5,1.5)--(5.5,1.5)--(5.5,2)--(5,2)--(5,1.5);
\draw (5.25,1.75) node {$1$};
\filldraw [fill=gray!60!white] (6,0.5)--(6.5,0.5)--(6.5,1)--(6,1)--(6,0.5);
\draw (6.25,0.75) node {$1$};
\filldraw [fill=gray!60!white] (6.5,0)--(7,0)--(7,0.5)--(6.5,0.5)--(6.5,0);
\draw (6.75,0.25) node {$1$};
\filldraw [fill=gray!60!white] (6,0)--(6.5,0)--(6.5,0.5)--(6,0.5)--(6,0);
\draw (6.25,0.25) node {$1$};

\draw [very thick] (5,0) rectangle (7.5,3);
\draw [very thick] (5,0) rectangle (6,1.5);

\draw [dashed, gray!80!white] (5,3)--(7.5,0);

\draw (6.25,-0.4) node {$V(D)$};

\end{tikzpicture}
\caption{The left picture represents the set $U(D)$ with arrows representing the bijection $\varphi.$ The right picture represents the set $V(D).$ Numbers represent the splittings of $U(D)$ and $V(D)$ into subsets $U_1,U_2,U_3$ and $V_1,V_2,V_3$ correspondingly.}\label{exampleUV}
\end{center}
\end{figure}
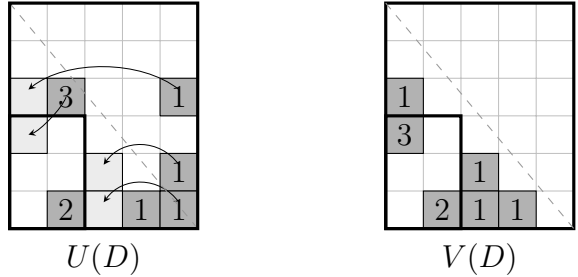

\end{example}

\section{$q,t$-Catalan Numbers and Poincar\'e Polynomials.}
\label{sec:poly}

When $(p,q)=(n,n+1)$, the statistic $\hnn(D)$ is also called $\dinv(D)$.

In \cite{gaha} A. Garsia and M. Haiman constructed the sequence of polynomials $C_{n}(q,t)$. These polynomials are symmetric in the variables $q$ and $t.$ They also unified two previously known one-parameter generalizations of the Catalan numbers:
\begin{equation}
C_n(1, t)=\sum_{D\subset \rnn}t^{\binom{n}{2}-\area(D)},\quad q^{\binom{n}{2}}C_{n}(q,1/q)=\frac{1}{[n+1]_{q}}\binom{2n}{n}_{q},
\end{equation}
\noindent Here we used the standard notation:
$$[k]_q=(1-q^k)/(1-q), [k]_q!=[1]_q[2]_q\cdots[k]_q, \binom{n}{k}_q=\frac{[n]_q!}{[k]_q![n-k]_q!}.$$
The relation of $q,t$-Catalan numbers to the geometry of the Hilbert scheme of points on $\mathbb{C}^2$ was discovered in \cite{haiman}.

In  \cite{gahaim}, \cite{haglund} the  following formula for the $q,t$-Catalan numbers was proved:
\begin{equation}
C_n(q, t)=\sum_{D\subset \rnn}q^{\dinv(D)}t^{\binom{n}{2}-|D|}.
\end{equation}

\begin{corollary}
If $(p,q)=(n,n+1)$, then $\dim\Delta_{D}=\binom{n}{2}-\dinv(D).$ Therefore,
\begin{equation}
\sum_{\Delta\in \Mod_{\gnn}}q^{\binom{n}{2}-\dim(\Delta)}t^{|\ZZ\backslash \Delta|}=C_{n}(q,t)
\end{equation}
\end{corollary}

\begin{lemma}
\label{arean}
The Poincar\'e polynomial of the Jacobi factor of an irreducible plane curve singularity with one Puiseux pair $(n,n+1)$ is equal to
$\sum\limits_{D\subset \rnn}t^{2|D|}.$
\end{lemma}

\begin{proof}
By Theorem \ref{pi}  Poincar\'e polynomial is equal to
$\sum_{\Delta\in \Mod_{\gnn}}q^{2\cdot \dim(\Delta)}.$ On the other hand,
$$\sum_{\Delta\in \Mod_{\gnn}}q^{\binom{n}{2}-\dim(\Delta)}=C_{n}(q,1)=C_{n}(1,q)=\sum_{D\subset \rnn}q^{\binom{n}{2}-|D|}.\qedhere$$
\end{proof}

It is natural to conjecture that the direct generalization of Lemma \ref{arean} is true for arbitrary coprime $(p,q)$. Set again $\Gamma=\gpq,$\ $R=R^{p,q},$ and $R_+=\rpq.$
One can associate another Young diagram to a semi-module $\Delta\in \Mod_{\gpq}:$

\begin{definition}\label{D'}
Let $D'(\Delta)$ be the diagram with columns
$$g(a_j)=|[a_j, a_j + q)\setminus \Delta|,$$
where  $0=a_0<a_1<\ldots<a_{p-1}$ is the $p$-basis of the $\Gamma$--semi-module $\Delta$.
\end{definition}

\begin{remark}
Note that $g(a_j)=|[a_j, a_j + q)\setminus \Delta|=|(a_j, a_j +
q]\setminus \Delta|,$ because both $a_j$ and $a_j+q$ are elements of
$\Delta.$ We will be using both formulas.
\end{remark}

By Definition \ref{def of dim}, we have $\dim\Delta=|D'(\Delta)|$. We illustrate
the Definition \ref{D'} in Figure \ref{5times7_2}.

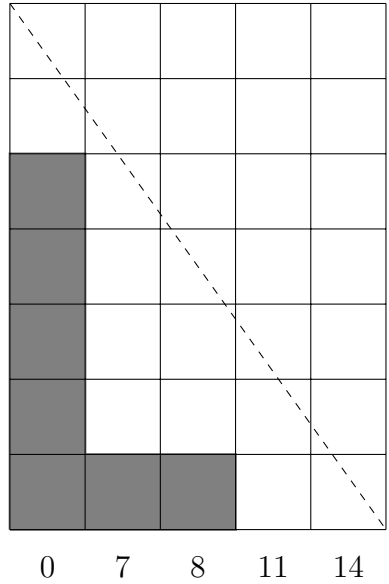
\begin{figure}[ht]
\begin{center}
\begin{tikzpicture}
\filldraw [color=gray, line width=1]
(0,0)--(0,5)--(1,5)--(1,1)--(3,1)--(3,0)--(0,0);
\draw (0,0) grid (5,7);

\draw (0.5,-0.5) node {$0$};
\draw (1.5,-0.5) node {$7$};
\draw (2.5,-0.5) node {$8$};
\draw (3.5,-0.5) node {$11$};
\draw (4.5,-0.5) node {$14$};

\draw [dashed] (0,7)--(5,0);
\end{tikzpicture}
\caption{The gray boxes form the diagram $D'(\Delta),$ where $\Delta$
is from Figure \ref{5times7}. Recall that the $5$-basis is
$\{0,7,8,11,14\}.$ One can check that $g(0)=5,\ g(7)=g(8)=1,$ and
$g(11)=g(14)=0.$ Therefore, $\dim\Delta=7.$}\label{5times7_2}
\end{center}
\end{figure}




\begin{lemma}\label{D' subset R+}
For any $\Gamma$-semi-module $\Delta$ the Young diagram $D'(\Delta)$ can be inscribed in $R_+.$
\end{lemma}

\begin{proof}
The $k$th column of $R_+$ has height $q-\left\lceil\frac{kq}{p}\right\rceil.$ Therefore we need to prove the inequalities $g(a_{k-1})\le q-\left\lceil\frac{kq}{p}\right\rceil,$\ $k=1,\dots p,$ where $a_0<a_1<\dots<a_{p-1}$ is the $p$-basis of $\Delta.$

Consider the interval $(a_{k-1},a_{k-1}+q].$ For every $p$-generator $a_m$ such that $m\le k-1$ there are at least $\left\lfloor\frac{q}{p}\right\rfloor$ integers congruent to $a_m$ on this interval. Since $a_m\le a_{k-1}$ and $a_m\in \Delta$ one gets the following estimate:
$$
g(a_{k-1})\le q-k\left\lfloor\frac{q}{p}\right\rfloor.
$$
However, this is clearly not good enough for us. One can improve this estimate in the following way. Let $a_m$ and $a_l$ be $p$-generators of $\Delta$ such that $l,m\le k-1$ and $a_l-a_m\equiv q$ modulo $p.$ Then there are as many numbers congruent modulo $p$ to one of the generators $a_m$ or $a_l$ in the interval $(a_{k-1},a_{k-1}+q]$, as there are integers congruent to $a_m$ in the interval $(a_{k-1},a_{k-1}+2q].$ Indeed, $x\equiv a_l$ modulo $p$ iff $x+q\equiv a_m$ modulo $p.$  

Therefore, we get at least $\left\lfloor\frac{2q}{p}\right\rfloor$ integers in the interval $(a_{k-1},a_{k-1}+q],$ congruent to $a_m$ or $a_l$ modulo $p,$ which is a better estimate compare to $2\left\lfloor\frac{q}{p}\right\rfloor.$ More generally, one gets the following Lemma:

\begin{lemma}
\label{numbers}
Let $m<p$ be a positive integer. Then for any $A\in\mathbb Z$ there are at least $\left\lfloor\frac{mq}{p}\right\rfloor$ integers in the interval $I=(A,A+q],$ congruent to $0,$\ $q,$\ $\dots,$\ $(m-2)q,$ or $(m-1)q$ modulo $p.$
\end{lemma}
 
To apply Lemma \ref{numbers} in full strength, one needs to split the $p$-generators $a_0,\dots, a_{k-1}$ into groups as follows. Two generators $a_m$ and $a_l,\ m,l\le k-1,$ such that $a_l-a_m\equiv q$ modulo $p$ belongs to the same group. We extend this relation to an equivalence relation on the set $a_0,\dots, a_{k-1}.$ Let $B_1,\dots, B_n$ be the equivalence classes. For every $0\le i\le n$ the generators from the class $B_i$ can be put in the following order:
$$
B_i=\{a^i_1,\dots,a^i_{k_i}\},\ a^i_{j+1}-a^i_j\equiv q\ (\mbox{modulo}\ p)\ \mbox{for}\ 0\le j<k_i.
$$

\begin{remark}
Note that $g(a_{p-1})=0\le q-\left\lceil\frac{pq}{p}\right\rceil=0.$ Therefore, it is enough to consider $k<p,$ in which case the above ordering of elements of the classes $B_i$ is uniquely defined. 
\end{remark}

Therefore, applying Lemma \ref{numbers} one gets the following estimate:  
$$
g(a_{k-1})\le q-\sum\limits_{i=1}^n \left\lfloor\frac{|B_i|q}{p}\right\rfloor.
$$
Finally, one can further improve this estimate by $n$ in the following way. For every $1\le i\le n$ one gets $a^i_{k_i}+q\in \Delta.$ Note that $a^i_{k_i}+q>a_{k-1},$ because otherwise there would be a $p$-generator congruent to $a^i_{k_i}+q$ in $B_i,$ which would contradict with the way we ordered the elements of $B_i.$ Because of the same reason, $a^i_{k_i}+q$ is not congruent to any of the generators $a_0,\dots,a_{k-1}$ modulo $p.$ Therefore, one gets 
$$
g(a_{k-1})\le q-\left(\sum\limits_{i=1}^n \left\lfloor\frac{|B_i|q}{p}\right\rfloor\right)-n=q-\sum\limits_{i=1}^n \left(\left\lfloor\frac{|B_i|q}{p}\right\rfloor+1\right)=
$$ 
$$
=q-\sum\limits_{i=1}^n \left\lceil\frac{|B_i|q}{p}\right\rceil\le q-\left\lceil\frac{kq}{p}\right\rceil.
$$

\end{proof}

\begin{example}
We illustrate the proof of Lemma \ref{D' subset R+} on the following example. Let $p=5$ and $q=8.$ Suppose that $0,1,2\in \Delta.$ We immediately get $a_0=0,\ a_1=1,$ and $a_2=2.$ Let us estimate $g(a_2).$ Our goal is to show that 
$$
g(a_2)\le q-\left\lceil\frac{3q}{p}\right\rceil=8-5=3.
$$ 
We will actually show that in this case $g(a_2)\le 2.$

By definition, we need to show that there are at least $5$ elements of $\Delta$ in the interval $[3,10].$ The first observation is that each of the generators $0,1,$ and $2$ gives rise to at least $\left\lfloor\frac{8}{5}\right\rfloor=1$ integer in the interval, congruent to the corresponding generator modulo $5.$ However, this is not enough -- it only gives us $3$ elements of $\Delta$ in the interval.

Next we want to use Lemma \ref{numbers}. To do that we split the set $\{0,1,2\}$ in two subsets:
$$
B_1=\{2,0\},
$$
$$
B_2=\{1\}.
$$
Note that $0-2\equiv 8$ modulo $5.$ According to Lemma \ref{numbers}, in the interval $[3,10]$ we should have at least one integer congruent to $1,$ and at least $\left\lfloor\frac{2\times 8}{5}\right\rfloor=3$ integers congruent to $0$ or $2.$ Indeed, $6$ is congruent to $1,$ $5$ and $10$ are congruent to $0,$ and $7$ is congruent to $2.$

Finally, we have $0+8=8\in\Delta\cap [3,10]$ and $1+8=9\in\Delta\cap [3,10],$ and neither of them is congruent to $0,1,$ or $2.$ So we got $6$ integers $5,6,7,8,9,10\in\Delta\cap [3,10].$ Therefore $g(a_2)\le 8-6=2.$
\end{example}

We know that the map $D:\Delta \mapsto D(\Delta)$ provides a bijection between the $\Gamma$-semi-modules and the Young diagrams inscribed in $R_+.$ Therefore, one can consider the composition $G=D'\circ D^{-1}.$

\begin{conjecture}
The map $G$ is a bijection from the set of subdiagrams of $R_{+}$ to itself. In particular,
the Poincar\'e polynomial of the Jacobi factor is equal to
\begin{equation}
\label{poin}
P_{p,q}(t)=\sum_{D\subset R_{+}}t^{2\dim\Delta(D)}=\sum_{D\subset R_{+}}t^{2|G(D)|}=\sum_{D\subset R_{+}}t^{2|D|}
\end{equation}
\end{conjecture}

The equation (\ref{poin}) agrees with the tables presented in \cite{piont}. In Theorem \ref{nnp1} we prove this conjecture for the case $(p,q)=(n,n+1).$ In this case the map $G$ should be compared with the bijection constructed by J. Haglund in \cite{haglund}. In the Appendix we illustrate the bijection $G$ for $p=3,\ q=4.$

\subsection{Bijectivity}

In the $(n,n+1)$ case we would like to present the explicit proof of the bijectivity of the map $G$.
This proof can be compared with the bijections constructed by Haglund and Loehr (\cite{hagl1},\cite{haglund},\cite{loehr})
to match the $\dinv,\area$ and $\bounce$ statistics.

As before, let $0=a_0<a_1<\dots<a_{n-1}$ be the $n$-basis of a semi-module $\Delta.$

\begin{lemma}
The number of $n$-generators of $\Delta$ in the interval $[a_{i}+n,a_{i+1}+n]$ equals $g(a_i)-g(a_{i+1})$.
\end{lemma}

\begin{proof}
Indeed, a number $x$ is an $n$-generator of $\Delta$ if and only if $x\in \Delta$ and $x-n\notin \Delta$, so
the number of $n$-generators equals 
$$|[a_{i}+n,a_{i+1}+n]\cap \Delta|-|[a_i,a_{i+1}]\cap \Delta]|=|[a_{i},a_{i}+n)\setminus \Delta|-|[a_{i+1},a_{i+1}+n)\setminus \Delta|=$$ $$g(a_i)-g(a_{i+1}).$$
\end{proof}

\begin{lemma}
In the $(n,n+1)$ case we have $g(a_i)=g(a_{i+1})$ iff $[a_{i},a_{i+1}]\subset \Delta$.
\end{lemma}

\begin{proof}
Let  $g(a_i)=g(a_{i+1})$. Let $k$ be the maximal integer such that $[a_i,k]\subset \Delta$ and suppose that $k<a_{i+1}$.
Since $g(a_i)-g(a_{i+1})$ equals  the number of $n$-generators in $[a_{i}+n,a_{i+1}+n]$, there are no such generators.
On the other hand, $k+n+1$ is an element of $\Delta$ in  $[a_{i}+n,a_{i+1}+n]$ which cannot be a $n$-generator,
so $k+1\in \Delta$. Contradiction.

Let   $[a_{i},a_{i+1}]\subset \Delta$, in this case there are no $n$-generators in $[a_{i}+n,a_{i+1}+n]$.
\end{proof}

\begin{lemma}
\label{sk}
In the $(n,n+1)$ case, let $s_k$ be the minimal $n$-generator of $\Delta$ in the interval $[kn,(k+1)n]$.
Then $[kn,s_k]\subset \Delta$.
\end{lemma}

\begin{proof}
Let $l_k$ be the maximal number such that $[kn,l_k]\subset \Delta$. Suppose that $l_k<s_k$, then $l_k$ is not a $n$-generator,
so $l_k-n\in \Delta$, so $l_k+1=(l_k-n)+(n+1)\in \Delta$. Therefore $l_k$ is not maximal. Contradiction.
\end{proof}

\begin{lemma}
\label{freak}
Let $\alpha,\beta$ be $n$-generators of $\Delta$, $\beta-\alpha\equiv 1 \left(\mod n\right)$.
Then $\beta\le \alpha+n+1$.
\end{lemma}

\begin{proof}
Follows from the fact that $\beta$ and $\alpha+n+1$ have the same remainder modulo $n$.
\end{proof}

\begin{theorem}
\label{nnp1}
The map $G$ is bijective in the $(n,n+1)$ case.
\end{theorem}

\begin{remark}
The map $G$ coincides with the bijection from \cite{hagl1},\cite{haglund} transforming a pair of statistics $(\area,\dinv)$ into 
$(\bounce,\area)$ statistics. We further explore this connection in our next paper \cite{GM}.
\end{remark}

\begin{proof}
We will describe the method of reconstruction of  the diagram $D$ (or, equivalently,
the $n$-generators $a_i$) from the diagram $G(\Delta)$.
One can compare this algorithm with the construction in Theorem 3.15 in \cite{haglund}.

Remark that $$n-g(a_j)=|[a_j, a_j + n)\cap \Delta|$$

{\bf Step 1.} For all $k$ the number $m_{k}$ of generators $a_i$ in $[kn,(k+1)n)$ is determined by $G(\Delta)$.

Let us prove it by induction by $k$. First, $m_0=n-g(0)$.
Suppose that we already know $m_0,\ldots, m_k$. Then by Lemma \ref{sk} we have $$[(k+1)n,(k+1)n+s_{k+1}]\subset \Delta$$
and $s_{k+1}=a_{m_0+\ldots+m_k}.$ Therefore $|[(k+1)n,(k+2)n)\cap \Delta|=n-g(s_{k+1})$, and
$$m_{k+1}=n-g(s_{k+1})-m_0-\ldots-m_k=n-g(a_{m_0+\ldots+m_k})-m_0-\ldots-m_k.$$
Note that this step corresponds to the construction of the "bounce path" in \cite{haglund}.

{\bf Step 2.} For all $k$ the order of remainders modulo $n$ of generators $a_i$ in $[kn,(k+2)n[$ is determined by $G(\Delta)$.

Suppose that $a_j\in [kn,(k+1)n), a_l\in [(k+1)n,(k+2)n)$. Then $n-g(a_j)$ equals  the number of generators  less than $a_j+n$,
therefore $a_l<a_j+n$ iff $l<n-g(a_j)$.

{\bf Step 3.} Given $G(\Delta)$, the set of generators $a_i$ can be reconstructed.

By the previous step, we can sort the generators in $[0,2n]$ and in $[n,3n]$ with respect to their remainders modulo $n$.
We want to merge these sequences of generators: we would like to understand how the generators from $[2n,3n]$ can be fitted in the sequence of generators from $[0,2n]$.

Suppose that we have two generators $\alpha\in [0,n]$, and $\beta\in [2n,3n].$ One can show that the remainder of $\beta$ is bigger than the remainder of $\alpha$ iff
there exists a generator $\gamma\in [n,2n],$ whose remainder is greater than the remainder of
$\alpha,$ and less than the remainder of $\beta$. Indeed, consider the generator $\delta$ congruent to $\alpha+1$ modulo $n$. By Lemma \ref{freak}, $\delta<2n.$ If $\delta>n,$ we are done. If $\delta<n,$ consider the generator congruent to $\alpha+2$ and repeat the argument.

Therefore, we can merge the sequences in a unique way. By repeating this procedure inductively,
we can reconstruct the remainders of all $a_i$ modulo $n$.

\end{proof}

\section{Subvariety of the Hilbert scheme of points}\label{Mpqh}

\begin{definition} Let $V_{p,q}$ denote the vector subspace of $\mathbb{C}[x, y]$ spanned by the monomials $x^{i}y^j$ such that $pi+qj<(p-1)(q-1)$.
Let $\mathcal{M}_{p,q,h}$ be the subset of the Hilbert scheme of points $\Hilb^{h}(\mathbb{C}^2)$
parametrising ideals $I\subset \mathbb{C}[x, y]$ such that $I + V_{p,q} = \mathbb{C}[x, y]$.
\end{definition}

The family of varieties $\mathcal{M}_{p,q,h}$ for $(p,q)=(n,kn+1)$ was considered by A. Buryak  (Theorem 1.5 in \cite{buryak}). By the construction, for all $h$ the varieties $\mathcal{M}_{p,q,h}$ are smooth subvarieties of $\Hilb^{h}(\mathbb{C}^2)$.

\begin{theorem}
The virtual Poincar\'e polynomial of the variety $\mathcal{M}_{p,q,h}$ is given by the formula (compare with Theorem \ref{piontcell})
$$P_{t}(\mathcal{M}_{p,q,h})=\sum_{D\subset R_{+}, |D|=h}t^{2(h+\hpq(D))}$$
\end{theorem}

\begin{proof}
Consider the action of $T=(\mathbb{C}^{*})^2$ on $\Hilb^{h}(\mathbb{C}^2)$. Fixed points of this action correspond to
monomial ideals in $\mathbb{C}[x, y]$. Let $I\subset C[x, y]$ be a monomial
ideal of colength $h$. Let $D_I = \{(i, j)\in \ZZ^2|x^{i}y^{j}\not \in I\}$ be the
corresponding Young diagram. Let $R(T) = \mathbb{Z}[t_1, t_2]$ be the representation ring of $T$. Then the
weight decomposition of $T_{I}\Hilb^{h}(\mathbb{C}^2)$ is given by the following formula (\cite{nakajima},\cite{elstro}):
\begin{equation}
T_{I}\Hilb^{h}(\mathbb{C}^2)=\sum_{c\in D}(t_1^{l(s)+1}t_2^{-a(s)}+t_1^{-l(s)}t_2^{a(s)+1}).
\end{equation}

Consider the action of the one-parameter $(p,q)$-subgroup of $T$ on $\mathcal{M}_{p,q,h}$. Such an action defines a
cell decomposition of $\mathcal{M}_{p,q,h}$ by unstable varieties of its fixed points (\cite{bibi},\cite{bibi2}), since for every $I\in \mathcal{M}_{p,q,h}$
the limit $\lim_{t\to 0}t\cdot I$ belongs to $\mathcal{M}_{p,q,h}$.

The fixed points
correspond to the Young diagrams contained in $R_{+}$ of area $h$, and the unstable subspace at a point $I$ has dimension
$$\dim_{D}=~\vline \left\{s\in D_{I}:\frac{a(s)}{l(s)+1}<\frac{p}{q}\right\}\vline~+~\vline \left\{s\in D_I :\frac{a(s)+1}{l(s)} > \frac{p}{q}\right\}\vline=|D_I|+\hpq(D).$$

Remark that since $p$ and $q$ are coprime, we can never get the equality of one of the above fractions to $\frac{p}{q}$.

As a conclusion, the variety $\mathcal{M}_{p,q,h}$ admits a cell decomposition where cells are labelled by diagrams $D\subset R_{+}$
with $|D|=h$ and the dimension of the cell corresponding to $\Delta$ is $h+\hpq(D)$.
\end{proof}

\section{Remarks}
\label{sec:knots}

Our work was partially inspired by the emerging development of the algebraic and geometric models for the 
Khovanov-Rozansky homology \cite{KR} of torus knots. Given a plane curve singularity, one can consider its intersection with a small
3-sphere, which is a link in $S^3$. If a singularity has one Puiseux pair $(p,q)$, then its link is isotopic to the $(p,q)$ torus knot.

Recently, two approaches to understanding the conjectural structure of the HOMFLY homology of torus knots were proposed. In \cite{qtcat}
the generators in HOMFLY homology of a $(p,q)$ torus knot were suggested to be enumerated by the lattice paths with marked corners in the $p\times q$
rectangle above the diagonal. For $q=p+1$ this construction counts the Schr\"oder paths in a square, and it was conjectured that the gradings of generators are given by certain combinatorial statistics introduced by J. Haglund, and that the bigraded Poincar\'e polynomial is related to the $q,t$-Catalan numbers of A. Garsia and M. Haiman. Of course, all these combinatorial constructions are finite, so we regard this theory as the reduced version of the HOMFLY homology.

Another approach was proposed in \cite{oshe} and developed in \cite{ORS}. It was conjectured that the unreduced HOMFLY homology of the link of a plane curve singularity $C$ corresponds to the homology of certain strata in the Hilbert scheme of points on $C$. It was also conjectured in \cite{ORS} that the reduced HOMFLY homology can be computed in terms of a certain filtration on the cohomology of the compactified Jacobian of $C$. 

The results of the present article suggest some similarity between the combinatorial structures in these approaches: the cells of the compactified Jacobian of $C$ 
are parametrized by the lattice paths, and Theorem \ref{piontcell} relates the dimension of a cell to the above combinatorial statistics.
We refer the reader to \cite{ORS} for more details.

\newpage
\section*{Appendix}

For $(p,q)=(3,4)$ we list all objects appearing in the text: $\Gamma^{3,4}$--semi-modules $\Delta$, 
the corresponding diagrams $D(\Delta)$, $3$-generators and $4$-cogenerators of $\Delta$, and the dual diagrams $D'(\Delta)$.

\begin{center}
\begin{figure}[ht]
\begin{tikzpicture}
\draw (0,0)--(0,13.5)--(12,13.5)--(12,0)--(0,0);
\draw (3.5,0)--(3.5,13.5);
\draw (5.3,0)--(5.3,13.5);
\draw (7.5,0)--(7.5,13.5);
\draw (10.3,0)--(10.3,13.5);
\draw (0,2.5)--(12,2.5);
\draw (0,5)--(12,5);
\draw (0,7.5)--(12,7.5);
\draw (0,10)--(12,10);
\draw (0,12.5)--(12,12.5);


\draw (2,13) node {\small $\Delta$};
\draw (4.5,13) node {\small $D(\Delta)$};
\draw (6.4,13) node {\small $3$-generators};
\draw (8.9,13) node {\small $4$-cogenerators};
\draw (11,13) node {\small $D'(\Delta)$};



\draw (1.7,11) node {\small $\{0,3,4,6,\ldots\}$};


\draw (3.6,10.4) -- (3.6,12.4) -- (5.1,12.4) -- (5.1,10.4) -- (3.6,10.4);
\draw (4.1,10.4) -- (4.1,12.4);
\draw (4.6,10.4) -- (4.6,12.4);
\draw (3.6,10.9) -- (5.1,10.9);
\draw (3.6,11.4) -- (5.1,11.4);
\draw (3.6,11.9) -- (5.1,11.9);


\draw (3.8,10.2) node {\small $8$};
\draw (3.8,10.7) node {\small $5$};
\draw (3.8,11.2) node {\small $2$};
\draw (3.8,11.7) node {\small $-1$};
\draw (3.8,12.2) node {\small $-4$};
\draw (4.3,10.2) node {\small $4$};
\draw (4.8,10.2) node {\small $0$};


\draw [line width=1.5] (3.6,12.4)--(3.6,10.4)--(5.1,10.4);


\draw (6.5,11) node {\small $(0,4,8)$};


\draw (8.8,11) node {\small $(-4,-1,2,5)$};


\fill [color=gray] (10.5,11)--(10.5,12)--(11.0,12)--(11.0,11.5)--(11.5,11.5)--(11.5,11)--(10.5,11);
\draw (10.5,11)--(10.5,12)--(11.0,12)--(11.0,11.5)--(11.5,11.5)--(11.5,11)--(10.5,11);
\draw (10.5,11.5)--(11.0,11.5)--(11.0,11.0);



\draw (1.7,8.5) node {\small $\{0,3,4,5,6,\ldots\}$};


\draw (3.6,7.9) -- (3.6,9.9) -- (5.1,9.9) -- (5.1,7.9) -- (3.6,7.9);
\draw (4.1,7.9) -- (4.1,9.9);
\draw (4.6,7.9) -- (4.6,9.9);
\draw (3.6,8.4) -- (5.1,8.4);
\draw (3.6,8.9) -- (5.1,8.9);
\draw (3.6,9.4) -- (5.1,9.4);


\draw (3.8,8.2) node {\small $5$};
\draw (3.8,8.7) node {\small $2$};
\draw (3.8,9.2) node {\small $-1$};
\draw (3.8,9.7) node {\small $-4$};
\draw (4.3,8.2) node {\small $1$};
\draw (4.3,7.7) node {\small $4$};
\draw (4.8,7.7) node {\small $0$};


\draw [line width=1.5] (3.6,9.9)--(3.6,8.4)--(4.1,8.4)--(4.1,7.9)--(5.1,7.9);


\draw (6.5,8.5) node {\small $(0,4,5)$};


\draw (8.8,8.5) node {\small $(-4,-1,1,2)$};


\fill [color=gray]  (10.5,8.5)--(10.5,9.5)--(11.0,9.5)--(11.0,8.5)--(10.5,8.5);
\draw  (10.5,8.5)--(10.5,9.5)--(11.0,9.5)--(11.0,8.5)--(10.5,8.5);
\draw (10.5,9)--(11.0,9);



\draw (1.7,6) node {\small $\{0,2,3,4,5,6,\ldots\}$};


\draw (3.6,5.4) -- (3.6,7.4) -- (5.1,7.4) -- (5.1,5.4) -- (3.6,5.4);
\draw (4.1,5.4) -- (4.1,7.4);
\draw (4.6,5.4) -- (4.6,7.4);
\draw (3.6,5.9) -- (5.1,5.9);
\draw (3.6,6.4) -- (5.1,6.4);
\draw (3.6,6.9) -- (5.1,6.9);


\draw (3.8,6.2) node {\small $2$};
\draw (3.8,6.7) node {\small $-1$};
\draw (3.8,7.2) node {\small $-4$};
\draw (4.3,5.2) node {\small $4$};
\draw (4.8,5.2) node {\small $0$};
\draw (4.3,5.7) node {\small $1$};
\draw (4.3,6.2) node {\small $-2$};


\draw [line width=1.5] (3.6,7.4)--(3.6,6.4)--(4.1,6.4)--(4.1,5.4)--(5.1,5.4);


\draw (6.5,6) node {\small $(0,2,4)$};


\draw (8.8,6) node {\small $(-4,-2,-1,1)$};


\fill [color=gray]  (10.5,6)--(10.5,6.5)--(11.0,6.5)--(11.0,6)--(10.5,6);
\draw (10.5,6)--(10.5,6.5)--(11.0,6.5)--(11.0,6)--(10.5,6);



\draw (1.7,3.5) node {\small $\{0,1,3,4,5,6,\ldots\}$};


\draw (3.6,2.9) -- (3.6,4.9) -- (5.1,4.9) -- (5.1,2.9) -- (3.6,2.9);
\draw (4.1,2.9) -- (4.1,4.9);
\draw (4.6,2.9) -- (4.6,4.9);
\draw (3.6,3.4) -- (5.1,3.4);
\draw (3.6,3.9) -- (5.1,3.9);
\draw (3.6,4.4) -- (5.1,4.4);


\draw (3.8,3.2) node {\small $5$};
\draw (4.3,3.2) node {\small $1$};
\draw (4.8,3.2) node {\small $-3$};
\draw (3.8,3.7) node {\small $2$};
\draw (3.8,4.2) node {\small $-1$};
\draw (3.8,4.7) node {\small $-4$};
\draw (4.8,2.7) node {\small $0$};


\draw [line width=1.5] (3.6,4.9)--(3.6,3.4)--(4.6,3.4)--(4.6,2.9)--(5.1,2.9);


\draw (6.5,3.5) node {\small $(0,1,5)$};


\draw (8.8,3.5) node {\small $(-4,-3,-1,2)$};


\fill [color=gray]  (10.5,3.5)--(10.5,4)--(11.5,4)--(11.5,3.5)--(10.5,3.5);
\draw  (10.5,3.5)--(10.5,4)--(11.5,4)--(11.5,3.5)--(10.5,3.5);
\draw (11.0,3.5)--(11.0,4);



\draw (1.7,1) node {\small $\{0,1,2,3,4,5,6,\ldots\}$};


\draw (3.6,0.4) -- (3.6,2.4) -- (5.1,2.4) -- (5.1,0.4) -- (3.6,0.4);
\draw (4.1,0.4) -- (4.1,2.4);
\draw (4.6,0.4) -- (4.6,2.4);
\draw (3.6,0.9) -- (5.1,0.9);
\draw (3.6,1.4) -- (5.1,1.4);
\draw (3.6,1.9) -- (5.1,1.9);


\draw (3.8,1.2) node {\small $2$};
\draw (4.3,0.7) node {\small $1$};
\draw (4.8,0.7) node {\small $-3$};
\draw (4.3,1.2) node {\small $-2$};
\draw (3.8,1.7) node {\small $-1$};
\draw (3.8,2.2) node {\small $-4$};
\draw (4.8,0.2) node {\small $0$};


\draw [line width=1.5] (3.6,2.4)--(3.6,1.4)--(4.1,1.4)--(4.1,0.9)--(4.6,0.9)--(4.6,0.4)--(5.1,0.4);


\draw (6.5,1) node {\small $(0,1,2)$};


\draw (8.9,1) node {\small $(-4,-3,-2,-1)$};


\draw (11,1) node {\Large $\emptyset$};

\end{tikzpicture}
\caption{Semi-modules for the semigroup generated by $3$ and $4.$}
\end{figure}
\end{center}

\section*{Acknowledgements}

The authors would like to thank A. Oblomkov, V. Shende, A. Buryak , S. Gusein-Zade, S. Grushevsky  and J. Rasmussen for useful discussions.
E. G. was partially supported by the grants RFBR-08-01-00110-a, RFBR-10-01-00678, NSh-8462.2010.1 and the Dynasty fellowship for young scientists.


\end{document}